\documentclass{article}

\usepackage{amsmath}
\usepackage{amsfonts}
\usepackage{amssymb}
\usepackage{amsthm}
\usepackage{mathdots}
\usepackage{graphicx}
\usepackage{asymptote}
\usepackage{verbatim}

\newtheorem{lemma}{Lemma}
\newtheorem{theorem}{Theorem}

\newtheorem{proposition}{Proposition}
\theoremstyle{definition}
\newtheorem*{definition}{Definition}
\newtheorem*{notation}{Notation}
\newtheorem*{example}{Example}
\newtheorem*{remark}{Remark}

\title{Closure properties of predicates recognized by deterministic and non-deterministic asynchronous automata}
\author{Maria Monks \\ mm830@cam.ac.uk}

\newcommand{\AND}{\wedge}
\newcommand{\OR}{\vee}

\DeclareMathOperator{\id}{id}

\begin{document}

\maketitle{}

\begin{abstract}
  Let $A$ be a finite alphabet and let $L\subset (A^{\ast})^n$ be an $n$-variable language over $A$.  We say that $L$ is \textit{regular} if it is the language accepted by a synchronous $n$-tape finite state automaton, it is \textit{quasi-regular} if it is accepted by an asynchronous $n$-tape automaton, and it is \textit{weakly regular} if it is accepted by a non-deterministic asynchronous $n$-tape automaton.  We investigate the closure properties of the classes of regular, quasi-regular, and weakly regular languages under first-order logic, and apply these observations to an open decidability problem in automatic group theory. 
\end{abstract}

\section{Introduction}

A finite state automaton is a machine that reads a string of letters over some finite alphabet one letter at a time, and either accepts or rejects the string after reading it.  It has a finite set of internal states, some of which are designated ``start states,'' some of which are designated ``accept states,'' and some may be both or neither.  The string of letters is written on a tape that is fed to the machine.  The machine then starts in one of the start states, and upon reading the first letter, it changes to another state (possibly the same one) that depends on the current state and the letter being read.  It then moves to the second letter on the tape and repeats the process.  When it reaches the end of the string, if the machine is in an accept state, the string is accepted, and otherwise the string is rejected.

There are several different ways to generalize finite state automata to machines that read $n$ tapes at once for some $n\ge 1$, and either accept or reject the entire $n$-tuple of strings written on the tapes.  A \textit{synchronous} $n$-tape automaton reads all $n$ tapes simultaneously and at the same speed, reading the first letter of each tape, then the second letter on each tape, and so on, possibly changing states at each step.  An \textit{asynchronous} $n$-tape automaton reads from one string at a time, and its current state determines which tape it reads from next.  A \textit{non-deterministic asynchronous} $n$-tape automaton has its choice of several possible sets of tapes to read from at each step, and reads one letter from each of those tapes before moving to a next state and repeating the process.

The set of accepted strings or $n$-tuples of strings is called the \textit{language} accepted by the automaton.  In general, an $n$-variable language over a finite alphabet $A$ is any subset of $(A^\ast)^n$, where $A^{\ast}$ is the set of finite strings over $A$.  A language is \textit{regular} if it is accepted by a synchronous automaton, it is \textit{quasi-regular} if it is accepted by an asynchronous automaton, and it is \textit{weakly regular} if it is accepted by a non-deterministic asynchronous automaton.

Regular languages are well studied throughout the literature (see \cite{RabinScott} for an introduction to the topic), but quasi-regular and weakly regular languages are less understood.  In fact, several different definitions of each notion have appeared throughout the literature.  Furthermore, while the relations defined by regular languages are closed under first-order logical operators (union ($\vee$), intersection ($\wedge$), complementation ($\neg$), and projection ($\exists$)), this is not true of quasi-regular or weakly regular languages.  

In this paper, we investigate the properties of quasi-regular and weakly regular languages and their use in automatic group theory.  In section \ref{Unifying}, we unify several of the notions of asynchronous and non-deterministic asynchronous automata that have appeared throughout the literature.  In section \ref{Closure}, we investigate the closure properties of each class of languages under first order logical operators.  In section \ref{Groups}, we apply our results to a decidability problem in asynchronous automatic group theory posed in \cite{Epstein}.

\section{Classes of languages defined by automata}\label{Unifying}

\subsection{Finite state automata on strings}

  For any set $S$, let $P(S)$ denote the power set of $S$.  

\begin{definition}
  A (non-deterministic) \textit{finite state automaton} over an alphabet $A$ is a quadruple $(S,\Delta, S_0, S_f)$, where $S$ is a finite set of states, $\Delta:S\times (A\sqcup \{\epsilon\})\rightarrow P(S)$ is the \textit{transition function}, $S_0\subset S$ is the set of \textit{initial states}, and $S_f\subset S$ is the set of \textit{accept states}.
\end{definition}

The \textit{state diagram} of a non-deterministic finite state automaton $(S,\Delta,S_0,S_f)$ is the edge-labeled directed graph with vertex set $S$ and whose directed edges are the pairs $(s,t)$ of states for which $t\in \Delta(s,x)$ for some $x\in A\sqcup\{\epsilon\}$.  We label such an edge by the letter $x$.  We circle the accept states, and the remaining states are called \textit{failure states}.

  The \textit{language} $L(W)$ accepted by a finite state automaton $W$ is the set of all words $w=x_1x_2\cdots x_n$ for which there is a sequence of states $s_0,s_1,s_2,\ldots,s_n$ with $s_0\in S_0$, $s_n\in S_f$, and $s_{i}\in \Delta(s_{i-1},x_{i-1})$ for $i=1,\ldots,n$.  In terms of the state diagram, a word $w=x_1x_2\cdots x_n$ is accepted by $W$ (that is, $w\in L(W)$) if and only if there is a path of edges, starting at a start state and ending at an accept state, whose labels are $x_1,\ldots,x_n$ in that order.  

  It is often useful to consider the finite state automata for which the next state is completely determined by the current state and the letter being read.

\begin{definition} 
  A finite state automaton $(S,\Delta, S_0, S_f)$ over an alphabet $A$ is \textit{partial deterministic} if there is a unique start state $s_0\in S_0$ and for each state $s\in S$ and letter $x\in A$, we have that $|\Delta(s,x)|\le 1$ and $|\Delta(s,\epsilon)|=0$.  The automaton is \textit{deterministic} if $|\Delta(s,x)|=1$ for all $s\in S$ and $x\in A$.
\end{definition}

In terms of the state diagram, a partial deterministic automaton has a unique start state, no $\epsilon$-arrows, and at most one $x$-arrow starting from each state.  It is deterministic if there is exactly one $x$-arrow starting from each state.

Let $A^\ast$ denote the set of all (possibly empty) strings of letters in $A$.  A \textit{language} over $A$ is any subset of $A^\ast$.  Given languages $L$ and $M$ over $A$, let $L^\ast$ denote the set of all strings formed by concatenating a finite sequence of elements of $L$, and let $LM$ denote the language consisting of all strings of the form $lm$ where $l\in L$ and $m\in M$.  

\begin{definition}\label{regular}
The class of \textit{regular languages} over $A$ is the smallest class of languages over $A$ that:
\begin{itemize}
  \item contains the empty language, 
  \item contains the languages $\{x\}$ for each $x\in A$, and 
  \item is closed under ${}^\ast$, concatenation, union, and intersection.  
\end{itemize}
\end{definition}

A well known theorem by Kleene, Rabin, and Scott (see \cite{Epstein}, \cite{RabinScott}) states that all of these notions are equivalent:

\begin{theorem}[Kleene, Rabin, Scott]\label{KleeneRabinScott}
  Let $L$ be a language over an alphabet $A$.  The following are equivalent:
\begin{itemize}
  \item $L$ is accepted by a deterministic finite state automaton.
  \item $L$ is accepted by a partial deterministic finite state automaton.
  \item $L$ is accepted by a non-deterministic finite state automaton.
  \item $L$ is a regular language over $A$.
\end{itemize}
\end{theorem}


\subsection{Multi-tape automata}

The notion of a finite state automaton can be generalized to allow multiple tapes to be read by the machine simultaneously.  In this section, we follow the conventions in \cite{Epstein}.  To account for the fact that the word written on one tape may be longer than the word written on another, we introduce a \textit{padding symbol} $\$$ that we use to pad the shorter strings in order to obtain strings of the same length.  

\begin{definition}
  The $n$-tape \textit{padded alphabet} over $A$ is the set $$\left(A\sqcup\{\$\}\right)^n\backslash\{(\$,\$,\ldots,\$)\}.$$   We denote the $n$-tape padded alphabet by $A^\$$ when $n$ is understood.
\end{definition}

In order to distinguish between letters or words over $A$ and letters or words over $A^n$ or $A^\$$, we use the following conventions.  A Greek character such as $\sigma$ or $\mu$ will be used to denote a letter in $A\sqcup\{\$\}$, and we will use $a$, $b$, $c$, $d$, $x$, $y$, and $z$ to denote letters of $A$ (other than $\$$ or $\epsilon$).  The letters $u$, $v$, or $w$ will be used to denote a word over $A$, an overlined Greek character such as $\overline{\sigma}$ will be used to denote an $n$-tuple of letters $(\sigma_1,\ldots,\sigma_n)$, and an overlined English character such as $\overline{w}$ will be used to denote an $n$-tuple of words.  We use $|w|$ to denote the number of letters in $w$.

\begin{definition}
    Given an $n$-tuple of strings $\overline{w}=(w_1,\ldots,w_n)$, let $m=\max_i |w_i|$, and for $k=1,\ldots,m$ define $\overline{\sigma}_k$ to be the $n$-tuple consisting of the $k$th letters of each $w_i$, where the $k$th letter is taken to be a padding symbol $\$$ if $k>|w_i|$.  Then the \textit{padded string} associated with $\overline{w}$ is the word $\overline{\sigma}_1\cdots\overline{\sigma}_m$.  If $L$ is any set of $n$-tuples of strings over $A$, the associated \textit{padded extension} over $A^\$$, denoted $L^\$$, is the language consisting of the padded strings associated with the elements of $L$.
\end{definition}

\begin{example}
 The padded string associated with $(aa,abbc,cab)$ is $(aa\$\$,abbc,cab\$)$.
\end{example}

Notice that, given a language $K$ consisting of only padded strings, we may remove the $\$$ symbols to obtain the unique set $L$ of $n$-tuples of words over $A$ for which $K=L^\$$.

We now have the tools to define an $n$-tape finite state automaton.  

\begin{definition}
  An \textit{$n$-tape deterministic finite state automaton} over $A$ is a deterministic finite state automaton $W$ over the padded alphabet $A^\$$ that only accepts padded strings.  If the language $K$ consisting of the padded strings accepted by $W$ is equal to $L^\$$, then we say that $W$ accepts the language $L$, and that $L$ is an \textit{$n$-tape regular language} over $A$.
\end{definition}

Notice that, by Theorem \ref{KleeneRabinScott}, the definition of $n$-tape regular language agrees with Definition \ref{regular} in the case $n=1$. 


\subsection{Deterministic asynchronous automata}

The $n$-tape automata described above read all $n$ tapes in parallel, at the same speed.  For this reason, we say that such automata are \textit{synchronous}.  We now consider \textit{asynchronous automata}, which read from one tape at a time, and may switch tapes several times in the process.

Two equivalent definitions of two-tape automata have appeared in the literature independently.  The notion was first introduced in \cite{RabinScott}:

\begin{definition}
  A two-tape \textit{asynchronous automaton} over an alphabet $A$ is a deterministic automaton over the alphabet $A\sqcup\{\$\}$, along with a partition of the state set into two sets $S_L$ and $S_R$, called the \textit{left} and \textit{right} state sets.
\end{definition}

Define a \textit{shuffle} of an $n$-tuple of words $(w_1,\ldots,w_n)$ over an alphabet $A$ is an ordering of all the letters of $w_1,\ldots,w_n$ that respects the ordering in each of the words $w_i$.  For instance, two valid shuffles of $(abc,bd)$ are $abbcd$ and $badbc$ (the shuffle $abbcd$ also carries information about which $b$ came from the left or right component, but we use the term `shuffle' loosely to refer to either the mapping of the letters to their position, or to the word spelled by the shuffle).  Also, for any word $w$, define $w\$$ to be the word formed by appending the symbol $\$$ at the end of the string $w$.

The language $L(W)$ accepted by an asynchronous automaton $W$ is defined to be the set of all pairs $(u,v)$ of words over $A$ such that there is some (unique) shuffle of $(u\$,v\$)$ that is accepted by the underlying deterministic automaton and has the property that the automaton must be in a state in $S_L$ to read a letter from $u$ and in a state in $S_R$ to read a letter from $v$.

In \cite{Epstein}, asynchronous automata are defined as follows.

\begin{definition}
  A two-tape asynchronous automaton over an alphabet $A$ is a partial deterministic finite state automaton $W$ over the language $A\cup \{\$\}$, along with a partition of the set of states into five subsets $S_L$, $S_R$, $S_L^{\$}$, $S_R^{\$}$, and $S^{\$}$, such that the following hold:
\begin{itemize}
   \item $S^{\$}$ contains exactly one element, $s^{\$}$, which is also the unique accept state of the automaton.
   \item The start state of $W$ is in either $S_L$ or $S_R$.
   \item An arrow is labeled by $\$$ if and only if it maps a state in $X$ to a state in $Y$, where the pair $(X,Y)$ is one of $(S_L,S_R^{\$})$, $(S_R,S_L^{\$})$, $(S_L^{\$},S^{\$})$, or $(S_R^{\$},S^{\$})$. 
   \item Arrows starting in $S_L$ can only end in $S_L$, $S_R$, or $S_R^{\$}$.
   \item Arrows starting in $S_R$ can only end in $S_L$, $S_R$, or $S_L^{\$}$.
   \item Arrows starting in $S_L^{\$}$ can only end in $S_L^{\$}$ or $S^{\$}$.
   \item Arrows starting in $S_R^{\$}$ can only end in $S_R^{\$}$ or $S^{\$}$.
   \item No arrows start in $S^{\$}$.
\end{itemize}
\end{definition}

In this definition, the language accepted by an asynchronous automaton is the set of all pairs of words $(u,v)$ such that there is a (unique) shuffle of $(u\$, v\$)$ accepted by the underlying deterministic automaton.  Following a $\$$-arrow from, say, $S_L$ to $S_R^\$$ indicates that we have reached the end of the left tape and now only need to read the right tape until we reach another $\$$.

In order to distinguish between these two definitions, we call the former a \textit{semi-sorted} asynchronous automaton, and the latter a \textit{sorted} asynchronous automaton, since the states of the former are only sorted based on the tape being read, while the states of the latter are further sorted based on the number of $\$$ symbols the automaton has read so far.  

We can easily generalize each of these definitions to $n$ tapes.  For simplicity, we write $[n]$ to denote the set $\{1,2,\ldots,n\}$.

\begin{definition}
  An $n$-tape \textit{semi-sorted asynchronous automaton} over an alphabet $A$ is a partial deterministic finite state automaton over the alphabet $A\sqcup\{\$\}$, along with a partition of the state set into $n$ sets $S_1,\ldots,S_n$.
\end{definition}

\begin{definition}
  An $n$-tape \textit{sorted asynchronous automaton} over an alphabet $A$ is a partial deterministic finite state automaton $W$ over the language $A\cup \{\$\}$, along with a partition of the set of states into subsets of the form $S_{i}^V$ where $V$ is a proper subset of $[n]$ and $i\in [n]\backslash V$, and a final subset $S_f^{[n]}$, such that the following hold:
\begin{itemize}
   \item The start state of $W$ is in $S_i^{\emptyset}$ for some $i$.
   \item An arrow is labeled by $\$$ if and only if it maps a state in $X$ to a state in $Y$, where the pair $(X,Y)$ is of the form $(S_i^V,S_j^U)$ with $j\neq i$ and $U=V\cup i$.
   \item Arrows \textit{not} labeled by $\$$ that start in $S_i^V$ must end in $S_j^V$ for some $j\not\in V$.
   \item $S_f^{[n]}$ contains exactly one element, $s^{\$}$, which is also the unique accept state of the automaton.
   \item No arrows start in $S_f^{[n]}$.
\end{itemize}

\end{definition}

We show that these two definitions are equivalent.

\begin{theorem}
  Sorted and semi-sorted asynchronous automata accept the same class of languages.
\end{theorem}

\begin{proof}
  Let $W$ be an $n$-tape sorted asynchronous automaton with (partial) transition function $\Delta$ and with state sets $S_i^V$ and $S_f^{[n]}$ as in the definition. For $i=1,\ldots,n-1$, define $$T_i=\bigcup_V S_i^V$$ where $V$ ranges over the proper subsets of $[n]$ not containing $i$.  Also define $$T_n=\left(\bigcup_V S_n^V\right)\cup S_f^{[n]}$$ where $V$ ranges over the proper subsets of $[n]$ not containing $n$.  Then we see that the partition $\{T_i\}$ makes $W$ into a semi-sorted asynchronous automaton $M$ with $L(M)=L(W)$. 

  Conversely, let $M$ be an $n$-tape semi-sorted asynchronous automaton, with state sets $T_i$, $i=1,\ldots,n$.  We construct a sorted asynchronous automaton $W$ as follows.  We first construct $2^{n}-1$ exact copies of each $T_i$, labeled $S_i^V$ for each proper subset $V$ of $\{1,2,\ldots,n\}$, inheriting any arrows that start and end in $T_i$.  For any set $V$ and any two distinct indices $i,j\not\in V$, we draw arrows between states $s\in S_i^V$ and $t\in S_j^V$ if and only if the corresponding states in $T_i$ and $T_j$ are connected in $M$.  

The quality of being a start state or accept state is \textit{not} inherited, with one exception: if the start state of $M$ is in $T_i$, we define the corresponding element of $S_i^\emptyset$ to be the start state of $W$.   We also construct a new accept state $s^\$$ and define $S_f^{\{1,\ldots,n\}}=\{s^\$\}$.

  We now perform the following operations in order:

\begin{itemize}

  \item For each $i\in [n]$, let $V_i=[n]\backslash\{i\}$.  If an arrow labeled by $\$$ in $M$ starts at a state $s_i\in T_i$ and ends at an accept state of $W$, draw a new arrow in $W$ labeled by $\$$ from the corresponding state in $S_i^{V_i}$ to $s^\$$.

  \item  If an arrow labeled by $\$$ in $M$ starts in $T_i$ and ends in $T_j$, then for each $V$ not containing $i$ or $j$, draw a new arrow in $W$ labeled by $\$$ starting and ending at the corresponding states in $S_i^V$ and $S_j^{V\cup\{i\}}$.  

  \item Remove any $\$$-arrow in $W$ that both starts and ends in any of the sets $S_i^V$.

\end{itemize}

These operations guarantee that when we are done reading the $i$th tape and reach the corresponding $\$$-arrow, the next state is in some $S_j^V$ where $V$ contains $i$.  This makes the resulting automaton $W$ into a sorted asynchronous automaton that accepts the same language as $M$.  This completes the proof.
\end{proof}

We call a language accepted by a (sorted or unsorted) asynchronous automaton a \textit{quasi-regular} language.  We shall see that the class of quasi-regular languages is a strict superset of the class of regular languages.  To illustrate this, we first prove a generalization of the well-known \textit{pumping lemma} for $n$-tape regular languages.

\begin{lemma}\label{npumping}
  Let $L$ be a regular $n$-variable language over an alphabet $A$.  There is a positive integer $p$ such that for any $n$-tuple of words $\overline{w}=(w_1,\ldots,w_n)\in L$ with $\max |w_i|\ge p$, there are nonnegative integers $k\ge 1$ and $l$, with $k+l\le m$, such that if we write each $w_i$ as $u_im_iv_i$ where $x_i$ consists of the $k$th through $(k+l)$th letters of $w_i$, then we have $$(u_1m_1^rv_1,\ldots, u_nm_n^rv_n)\in L$$ for all $r\ge 1$.  Moreover, each substring $x_i$ either consists entirely of letters in $A$ or consists entirely of $\$$ symbols.
\end{lemma}

\begin{proof}
  Let $W$ be an $n$-tape finite state automaton accepting $L$, and let $p$ be the number of states of $L$.  Then if $\overline{w}=(w_1,\ldots,w_n)$ is in $L$ such that $\max |w_i|\ge p$, the path of arrows traced out on $W$ that read $w$ visits at least $p+1$ states, and so some state must be visited more than once.  In particular, there is a loop of some length $l$ in the path, that starts at the $k$th arrow in the path.  For each $i=1,\ldots,n$, let $m_i$ denote the sequence of letters appearing in the $i$th coordinate along this loop.

  We may now form new accepted paths by repeating this loop $r$ times before continuing along the path.  Thus, if we write $w_i=u_im_iv_i$ then $\left(u_1m_1^rv_1,\ldots, u_nm_n^rv_n\right)$ is also accepted by $W$ for any $r\ge 1$.  

  Finally, if $m_i$ consists of some letters and some $\$$ symbols, we would obtain an accepted $n$-tuple of words which is not a padded string, which contradicts the definition of an $n$-tape finite state automaton.   Thus each $m_i$ either consists entirely of letters in $A$ or consists entirely of $\$$ symbols.
\end{proof}

We now provide an example demonstrating that not all quasi-regular languages are regular. \\

\begin{example}\label{QuasiNotRegular}
  The two-variable language $L=\{(x^{n},x^{2n})\mid n\in \mathbb{N}\}$ is quasi-regular but not regular.
\end{example}

\begin{proof}
  Since $L$ is accepted by the sorted asynchronous automaton shown in Figure \ref{sorted1}, $L$ is quasi-regular.

  Now, suppose $L$ were regular.  By Lemma \ref{npumping} there are nonnegative integers $k$ and $l$ such that we may repeat the $k$th to $(k+l)$th letters of each component any number of times to obtain new elements of $L$, as long as either $1\le k \le k+l\le n$ (when both subwords consist only of $x$'s) or $n+1\le k\le k+l\le 2n$ (when the left subword consists only of $\$$ symbols and the right consists only of $x$'s).  We consider these two cases separately.

  If $1\le k \le k+l\le n$, the words $\left(x^{n+r(l+1)},x^{2n+r(l+1)}\right)$ would be in $L$ for each $r\ge 0$, but since $l+1\ge 1$ we have $2(n+r(l+1))=2n+2r(l+1)\neq 2n+r(l+1)$ for $r>0$.  Thus these words are not in the language, a contradiction.

  If $n+1\le k\le k+l\le 2n$, the words $\left(x^{n}, x^{2n+r(l+1)}\right)$ would be in $L$ for $r\ge 0$, again a contradiction since $2n<2n+r(l+1)$ for $r> 0$.

  It follows that $L$ is not regular.
\end{proof}

\begin{figure}
\begin{center}
\includegraphics{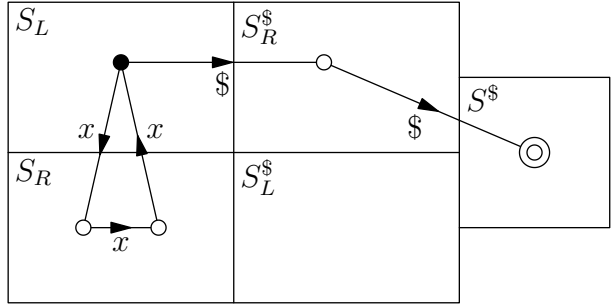}
\caption{\label{sorted1}A sorted asynchronous automaton that accepts the language $L=\{(x^{n},x^{2n})\mid n\in \mathbb{N}\}$.  The darkened circle indicates the start state, and the double circle indicates the accept state of the automaton.  The sets $S_L$, $S_R$, $S_L^\$$, $S_R^\$$, and $S^\$$ are outlined.}
\end{center}
\end{figure}


\subsection{Non-deterministic asynchronous automata}

  We now study \textit{non-deterministic} asynchronous automata, which may read from any of several possible subcollections of the tapes, called \textit{filters}, at each step, and has a choice of several possible next states at each transition.  In \cite{KN}, Khoussainov and Nerode defined these automata as follows.

\begin{definition}\label{FAA}  
Let $E=P\left(P([n])\backslash\{\emptyset\}\right)$, and call $E$ the set of \textit{filters} on $n$ tapes.  Let $A$ be a finite alphabet, and let $A^\$=\left(A\sqcup\{\$\}\backslash\{(\$,\$,\ldots,\$)\}\right)$ be the associated padded alphabet.  Then an \textit{$n$-tape non-deterministic filter asynchronous automaton}, or \textit{FAA}, is a quadruple $(S,S_0,\Delta, S_f)$ where:
\begin{itemize} 
  \item $S$ is a finite set of states,
  \item $S_0\subset S$ is the set of initial states,
  \item $S_f\subset S$ is the set of accept states, and
  \item $\Delta:S\times A^\$\rightarrow P(S)\times E$ is a transition function that, given a state and a letter over $A^\$$, returns a set of next states along with a set of filters, and
  \item for all $\overline{\sigma}=(\sigma_1,\ldots,\sigma_n) \in A^\$$ and for all $s\in S$, if $\sigma_i=\$$ and the set of next states given by $\Delta(s,\overline{\sigma})$ is nonempty, then $i$ is not in any of the filters given by $\Delta(s,\overline{\sigma})$.
\end{itemize}
\end{definition}

An $n$-tuple of words is accepted by a FAA if the following condition is satisfied.  We write the $n$ words in question on $n$ tapes, starting in a start state of $W$, we choose a valid filter as given by the transition function, move one position to the right along precisely those tapes whose index is in that filter, and non-deterministically choose one of the next possible states as the next state.  If this process can be repeated until the end of every tape is reached, and the final state of this process is an accept state, then the tuple is accepted by $W$.

Formally, an $n$-tuple of words $\overline{w}=(w_1,\ldots,w_n)$ is accepted by $W$ if and only if there is a sequence of states $s_0,s_1,\ldots,s_f$ where $s_0$ is a start state and $s_f$ is an accept state, along with an associated sequence of filters $\chi_0,\ldots,\chi_{f-1}$, with the following properties.  For each $k=0,\ldots,f$, let $\overline{\sigma_k}$ be the $n$-tuple whose $i$th coordinate is the $r$th entry of $w_i$, where $r$ is the total number of filters of $\chi_0,\ldots,\chi_k$ containing $i$.  Then if for all $k$, $\Delta(s_k,\overline{\sigma_k})=(S,X)$ where $S$ contains $s_{k+1}$ and $X$ contains $\chi_k$, the $n$-tuple $w$ is accepted by $W$.

\begin{remark}
  This is a slight modification of the original definition of Khoussainov and Nerode in \cite{KN}, which does not include the condition that the state set of $\Delta(s,\overline{\sigma})$ is nonempty in the last bullet point.  Note that, in a FAA, if $\Delta(s,\overline{\sigma})=(\emptyset, X)$ then we cannot make a move starting at $s$ with input $\overline{\sigma}$, and so the content of $X$ does not matter in determining its accepted language.  Thus, the two definitions are equivalent.  We use our modified version throughout.
\end{remark}

Another definition of non-deterministic asynchronous automata in the two-tape case appeared independently in \cite{Shapiro}.  Shapiro defined a two-tape non-deterministic asynchronous automaton to be a non-deterministic automaton along with a partition of the set of states into two sets. We may generalize Shapiro's definition to $n$ tapes as follows.

\begin{definition}
 An \textit{$n$-tape non-deterministic semi-sorted asynchronous automaton} (SAA) over an alphabet $A$ is a non-deterministic finite state automaton over $A\sqcup\$$ along with a partition of the set of states into $n$ sets $S_1,\ldots,S_n$.
\end{definition}

We say that an $n$-tuple of words $\overline{w}=(w_1,\ldots,w_n)$ is accepted by a SAA with transition function $\Delta$ if there is a shuffle $u$ of $(w_1\$,\ldots,w_n\$)$ and a sequence of states $s_1,\ldots,s_{|w_1|+\ldots+|w_n|+n}$ for which $\Delta(s_i,)$ on the diagram of the non-deterministic automaton, starting at a start state and ending at an accept state, that reads the tuple in question, where an arrow from a state in $S_i$ corresponds to reading and moving one position along the $i$th tape.

We will show that the class of languages ($n$-tuple relations) accepted by FAA's is identical to the class of languages accepted by SAA's.  In order to do so, we first define yet another type of automaton that accepts the same class of languages.

\begin{definition}
  A \textit{deterministic-filter (non-deterministic) asynchronous automaton}, or DFAA, is a FAA with the property that, in any given state, there is at most one possible filter to choose from.  In other words, if $\Delta(s,\overline{\sigma})=\left(S,X \right)$ then $|X|\le 1$.
\end{definition}

We now show that FAA's, DFAA's, and SAA's all have the same class of accepted languages.

\begin{theorem}
  Let $L$ be an $n$-variable language over a finite alphabet $A$.  The following are equivalent.
  \begin{itemize}
    \item $L$ is the accepted language of an FAA.
    \item $L$ is the accepted language of an DFAA.
    \item $L$ is the accepted language of a SAA.
  \end{itemize}
\end{theorem}

\begin{proof}
 Let $W=(S,S_0,\Delta, S_f)$ be a FAA.  We construct a DFAA $W'$ accepting the same language as $W$.

Define the state set of $W'$ to be 
$$S'=S\times (P([n])\backslash{\emptyset})=\{(s,\chi)\mid s\in S\text{ and }\chi\in P([n])\backslash\{\emptyset\}\}.$$  
Define the set of start states $S_0'$ to be the set of states $(s_0,\chi)\in S'$ such that $s_0\in S_0$, and define the set of final states $S_f'$ to be the set of states $(s_f,\chi)\in S'$ such that $s_f\in S_f$.  The transition function $\Delta'$ is given by 
$$\Delta'((s,\chi),\overline{\sigma})=
   \begin{cases}
     \left(\{(t,\mu)\mid t\text{ in the set of states of }\Delta(s,\overline{\sigma})\},\{\chi\}\right) & \text{if }\chi \text{ is a filter of }\Delta(s,\overline{\sigma}) \\ 
     \left(\{\},\{\chi\}\right) & \text{otherwise}
   \end{cases}.$$

In other words, the transition function for $W'$ takes in a pair $(s,\chi)$ and reads an $n$-tuple $\overline{\sigma}$.  If $\chi$ is a possible move of $W$ on state $s$ with input $\overline{\sigma}$, then $W'$ may move to any state $(t,\mu)$ where $t$ is a state that $W$ can reach from $s$ upon input $\overline{\sigma}$ and $\mu$ is any valid filter.  Furthermore, $W'$ moves along precisely those tapes whose index is in the filter $\chi$. 

We now show that $W'$ is a well-defined FAA; since there is a unique filter to choose from in any given state, it then follows that it is a DFAA.  

To show that $\Delta'$ is a valid FAA transition function, let $(s,\chi)$ be any state of $W'$ and let $\overline{\sigma}=(\sigma_1,\ldots,\sigma_n)$ be any $n$-tuple of letters over $A$.  First, suppose $\chi$ is in the set of possible filters of $\Delta(s,\overline{\sigma})$.  Then for any $i$ for which $\sigma_i=\$$, we have that $i\not \in \chi$ since $W$ is a FAA.  Thus $i$ does not occur in the set of possible filters, namely, $\{\chi\}$, of $\Delta'((s,\chi),\overline{\sigma})$.

Otherwise, if $\chi$ is not in the set of possible filters of $\Delta(s,\overline{\sigma})$, then the state set of $\Delta'((s,\chi),\overline{\sigma})$ is empty, and so the condition is trivially satisfied.

We now show that $W'$ accepts the same language as $W$.  Given an accepted $n$-tuple $w$ of words in $W$, there is a path of filters $\chi_1,\ldots, \chi_f$ that one follows from a start state $s_0$ to a final (accept) state $s_f$.  Let $\sigma^1,\ldots,\sigma^f$ be the $n$-tuples of letters that are read at each step along the way.  

Consider the path of states $(s_0,\chi_1), (s_1,\chi_2), \ldots, (s_f,\chi_f)$ in $W'$.  By our definition of $\sigma^1$, we may read $\sigma^1$ with filter $\chi_1$ to move from state $(s_0,\chi_1)$ to $(s_1,\chi_2)$, at which point we are reading $\sigma^2$, and so on, until we read all of $w$ and reach the accept state $(s_f,\chi_f)$.  Thus every word accepted by $W$ is accepted by $W'$. 

Conversely, suppose $(s_0,\chi_1), (s_1,\chi_2), \ldots, (s_f,\chi_f)$ is any path of states in $W'$, ending on an accept state $(s_f,\chi_f)$, that defines the sequence $\sigma^1,\ldots,\sigma^f$ of $n$-tuples of letters being read by the corresponding arrows.  Then there is a path between the states $s_0,\ldots, s_f$ with associated filters $\chi_1,\ldots,\chi_f$ that is accepted by $W$ and reads off precisely these $n$-tuples.   Thus every word accepted by $W'$ is accepted by $W$.

It follows that every language accepted by a FAA is also accepted by an DFAA.  Note that every DFAA is also a FAA by definition, and so every DFAA language is also accepted by a FAA.  This completes the first equivalence.

Now, let $V$ be an arbitrary DFAA. We construct a SAA $P$ that accepts the same language as $V$.  To do so, we first note that the states of $V$ can be sorted into sets based on their associated filter $\chi$. 

We can represent $V$ as a graph with the states as nodes and with arrows between states labeled by $n$-tuples of letters to indicate the transition diagram, where the nodes are sorted into $2^n-1$ disjoint sets, one for each filter $\chi$.  Let $S_i$ be the set whose filter consists only of the tape $i$.  For each state $T=(s,\chi)$ that is not in any $S_i$, we perform the following operation:

\begin{enumerate}
   \item Let $j_1,\ldots,j_k$ be the elements of the filter $\chi$ of $T$.   Then we move $T$ to the set $S_{j_1}$.  
   \item For each arrow starting at $T$, say $T\rightarrow T'$ labeled by $\sigma$, add new states $T_2,\ldots,T_k$ to the sets $S_{j_2},\ldots, S_{j_k}$ respectively, and draw arrows labeled by $\sigma$ from $T$ to $T_2$, from $T_2$ to $T_3$, etc., and then from $T_k$ to $T'$.  
\end{enumerate}

 Once this has been done, we replace the label $\sigma$ on any arrow starting in $S_i$ with the label $\sigma_i$, for it is only this letter which is allowed through.   It is clear that the resulting automaton $P$ accepts the same set of $n$-tuples of words as $V$.  It follows that every language accepted by a DFAA is also accepted by a SAA.

 Finally, given a SAA $V$, we may interpret it as a DFAA by making the associated filter of each state in $S_i$ be the filter $\{i\}$, and re-labeling the arrows starting in $S_i$ with $n$-tuples that match in the $i$th position for each $i$.  Thus every language accepted by a SAA is also accepted by a DFAA.  
\end{proof}


\section{Closure properties} \label{Closure}

A \textit{regular predicate} over an alphabet $A$ is any statement $P(x_1,\ldots,x_n)$ such that the set of tuples of words $(x_i)$ in $A^n$ for which $P$ holds is a regular language.  We can similarly define quasi-regular and weakly regular predicates.  It is known that regular predicates are closed under first-order predicate logic.  We now investigate the closure properties of quasi-regular and weakly regular predicates.

\begin{proposition}\label{QuasiClosure}
  In the following, let $P(x_1,\ldots,x_n)$ and $Q(x_1,\ldots,x_n)$ be $n$-variable quasi-regular predicates.
  \begin{enumerate}
    \item[(a)] The predicate $\neg P(x_1,\ldots,x_n)$ is quasi-regular.
    \item[(b)] The predicate $P(x_1,\ldots,x_n)\AND Q(x_1,\ldots,x_n)$ is not necessarily quasi-regular.  
    \item[(c)] The predicate $P(x_1,\ldots,x_n)\OR Q(x_1,\ldots,x_n)$ is not necessarily quasi-regular.  
    \item[(d)] The predicate $(\exists x_1)P(x_1,\ldots,x_n)$ is weakly regular, but not necessarily quasi-regular.
    \item[(e)] The predicate $(\forall x_1)P(x_1,\ldots,x_n)$ is not necessarily quasi-regular.
    \item[(f)] If $n=2$, the predicate $(\exists x_1)P(x_1,x_2)$ is regular.
    \item[(g)] If $n=2$, the predicate $(\forall x_1)P(x_1,x_2)$ is regular.
  \end{enumerate}
\end{proposition}

In summary, $n$-variable quasi-regular predicates (languages) are closed under $\neg$ (complementation), but not under $\vee$ (union), $\wedge$ (intersection), $\exists$ (projection) or $\forall$ (complementation of the projection of the complement).  In the case $n=2$, the application of $\exists$ or $\forall$ yields a $1$-variable regular language.

\begin{proof}
  See \cite{Epstein} for a proof of claims (a), (f), and (g).

  For (b), recall from Example \ref{QuasiNotRegular} that the language $\{x^n,x^{2n}\}$ is quasi-regular.  Similarly, the language $\{x^{2n},x^{n}\}$ is quasi-regular. We show that their union $L:=\{(x^n, x^{2n})\}\cup \{(x^{2n},x^n)\}$ is not quasi-regular.  

  Assume to the contrary that there is a semi-sorted deterministic asynchronous automaton $M$ accepting $L$.  Since $M$ has a finite number of states and the lengths of the paths accepting pairs of the form $(x^n,x^{2n})$ become arbitrarily large, there must exist a cycle in its state diagram.  Since no cycle can contain a $\$$ symbol, the cycle must consist entirely of edges labeled by $x$.  Tracing around this cycle will yield a word of the form $(x^s,x^t)$ for some $s$ and $t$.  

  Choose $N>0$ large enough so that the path accepting $(x^N,x^{2N})$ traverses this cycle at least once.  We can repeat the cycle $k$ times, so that $M$ accepts all words of the form $(x^{N+ks},x^{2N+kt})$ for nonnegative integers $k$.  It follows from the definition of $L$ that $t=2s>0$.

  Similarly, there exists a cycle of the form $(x^{u},x^{v})$ where $u=2v>0$.  These cycles are clearly distinct, and must occur on a path from the start vertex that does not contain any $\$$ symbols.  But since $M$ is deterministic, this is impossible, and we have a contradiction.

  To prove (c), assume to the contrary that $R\AND S$ is quasi-regular for any $n$-variable quasi-regular predicates $R$ and $S$.  Note that $P\OR Q$ is equivalent to $\neg(\neg P\AND \neg Q)$.  Since $\neg P$ and $\neg Q$ are quasi-regular, by our assumption we have that $\neg P\AND \neg Q$ is quasi-regular, and hence $\neg(\neg P\AND \neg Q)$ is quasi-regular as well.  Thus $P \OR Q$ is necessarily quasi-regular, contradicting (b).

  For (d), we first show that the predicate is weakly regular.  Let $M$ be a semi-sorted asynchronous automaton accepting the relation defined by $P$, with state sets $S_1,\ldots, S_n$.  Then we can replace all arrows starting in the state set $S_1$ corresponding to $x_1$ by $\epsilon$-arrows and merge the states of $S_1$ with $S_2$ to obtain a non-deterministic SAA that accepts $(\exists x_1) P(x_1,\ldots,x_n)$.

  To show $(\exists x_1) P(x_1,\ldots,x_n)$ is not necessarily quasi-regular, let $A=\{x,y,z\}$, and let $L=\{(y,x^n,x^{2n})\}\cup\{(z,x^{2n},x^n)\}$, where $n$ ranges over the nonnegative integers.  We show that the predicate $(a,b,c)\in L$ is a quasi-regular predicate over $A$, but its projection $\exists a, (a,b,c)\in L$ is not quasi-regular.  A semi-sorted asynchronous automaton accepting the language $L$ is shown in Figure \ref{sorted2}. 
\begin{figure}[h]
\begin{center}
\includegraphics{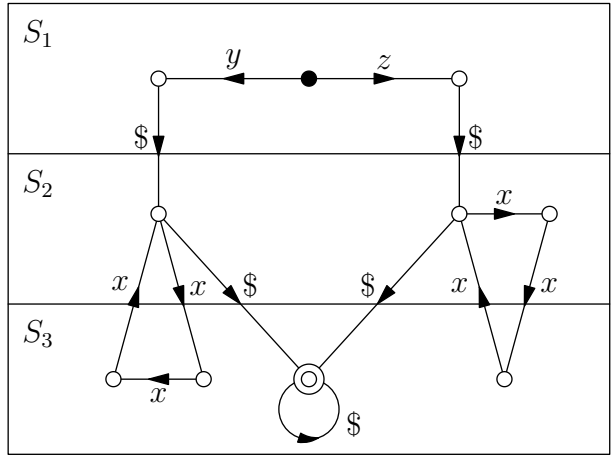}
\caption{\label{sorted2}A semi-sorted asynchronous automaton accepting the language $\{(y,x^n,x^{2n})\}\cup \{(z,x^{2n},x^n)\}$.}
\end{center}
\end{figure}
  Now, the predicate $\exists a, (a,b,c)\in L$ defines the two-variable language $\{(x^n,x^{2n})\}\cup\{(x^{2n},x^n)\}$, which is not quasi-regular, by our example for (b).

  For (e), we note that $(\exists x_1) P(x_1,\ldots,x_n)$ is equivalent to $\neg(\forall x_1) (\neg P(x_1,\ldots,x_n))$.  Thus, if $\forall$ maps quasi-regular predicates to quasi-regular predicates, it would follow that $\exists$ does as well by closure under complementation, contradicting (d).  Thus $\forall$ does not preserve quasi-regularity.
\end{proof}

  In part (d) of the above proposition, we found that applying the $\exists$ operator to a quasi-regular predicate yields a weakly regular predicate.   We now show that every weakly regular predicate can be obtained in this way.

\begin{theorem}\label{Bridge}
  Suppose $P(x_1,\ldots,x_n)$ is an $n$-variable weakly regular predicate.  Then there is an $(n+1)$-variable quasi-regular predicate $Q(x_0,\ldots,x_n)$ for which $$P(x_1,\ldots,x_n)\iff (\exists x_0)Q(x_0,\ldots,x_n).$$
\end{theorem}

\begin{proof}
  Let $L$ denote the language defined by $P(x_1,\ldots,x_n)$. Let $M$ be a non-deterministic semi-sorted asynchronous automaton (SAA) over an alphabet $A=\{\sigma_1,\ldots,\sigma_n\}$, with state sets $S_1,\ldots,S_n$, accepting the language $L$.  We construct from $M$ a semi-sorted asynchronous automaton $M''$, with an additional state set $S_0$, as follows.

  Let $k$ be the number of $\epsilon$-arrows appearing in the state diagram of $n$.  We choose any ordering of the $\epsilon$-arrows, and perform the following operation on the $i$th $\epsilon$ arrow for $i=1,\ldots,k$.  We create a new state $s_i$ in the new state set $S_0$, and make $s_i$ an accept state or start state if and only if $s$ is an accept state or start state, respectively.  Suppose the $\epsilon$ arrow begins in a state $r_1$ and ends in a state $r_2$ defined by $P(x_1,\ldots,x_n)$.  For each arrow $\alpha$ from any other state $t$ into $r_1$, we draw a new arrow with the same label as $\alpha$ from $t$ to $s_i$, and an arrow labeled by a new letter $\sigma_{n+i}$ from $s_i$ to $r_2$.  Then, we remove the $\epsilon$ arrow.

  We now have a new SAA $M'$ over an extended alphabet $\{\sigma_1,\ldots,\sigma_{n+k}\}$ having no $\epsilon$-arrows (here $k$ is the number of $\epsilon$-arrows in the original automaton $M$).  Note that we have simply re-routed every path through the original $\epsilon$-arrows with the use of extra letters appearing in the $0$th component, and so the $n$-tuples of words appearing as the last $n$ words in an $(n+1)$-tuple accepted by $M'$ are precisely those $n$-tuples in $L$.  Thus, $M'$ accepts a language $L'$ whose projection onto the last $n$ variables is the language $L$.

  Next, we modify $M'$ to form a semi-sorted asynchronous automaton $M''$, accepting another language $L''$ whose projection onto the last $n$ variables is also $L$.  Let $j$ be the total number of arrows $\alpha$ of $M$ such that the state $s$ at which $\alpha$ begins has at least one more arrow with the same label as $\alpha$ beginning at $s$.  (Notice that $s$ cannot lie in $S_0$, since in our construction above, every arrow starting in $S_0$ was given a unique label.)  Choose an ordering $\alpha_1,\ldots,\alpha_j$ of these arrows.  

  For each state $s$ having two arrows of the same label $\sigma$ beginning at $s$, we create a new state $s'$ in $S_0$.  We make $s'$ an accept state or start state if and only if $s$ is an accept state or start state, respectively.  Next, we draw an arrow labeled by $\sigma$ from $s$ to $s'$.  Now, each arrow labeled by $\sigma$ starting at $s$ is one of the arrows $\alpha_i$ by construction.  Suppose $\alpha_i$ ends at the state $s_i$.  We draw an arrow from $s'$ to $s_i$ labeled by a new letter $\sigma_{n+k+i}$, and we remove the arrow $\alpha_i$.  Notice that there is now exactly one arrow labeled $\sigma$ beginning at $s$, and by following the arrow into $S_0$, we can come out to any of the states that $\sigma$ originally pointed to in $M'$.  Thus, we have re-routed the redundant arrows through a single arrow into $S_0$, without changing any of the nonzero components of our accepted paths.

  We now have an automaton with no $\epsilon$-arrows and at most one arrow of each label starting from a given state.  Thus, to make it partial deterministic, we only need to consider the possibility that there are multiple start states.  Let $t_1,\ldots,t_m$ be the start states of the automaton.  We construct a new start state $r$ in $S_0$, and for each $t_i$ we draw an arrow from $r$ to $t_i$ labeled by a new letter $\sigma_{n+k+j+i}$.  We then make the states $t_i$ into non-start states.  This yields a semi-sorted asynchronous automaton $M''$, accepting a quasi-regular language $L''$, such that $(\exists x_0)(x_1,\ldots,x_n)\in L''$ defines the language $L$.
\end{proof}

\begin{proposition}\label{WeaklyClosure}
   In the following, let $P(x_1,\ldots,x_n)$ and $Q(x_1,\ldots,x_n)$ be $n$-variable weakly regular predicates.
\begin{enumerate}
  \item[(a)] The predicate $\neg P(x_1,\ldots,x_n)$ is not necessarily weakly regular.
  \item[(b)] The predicate $P(x_1,\ldots,x_n)\vee Q(x_1,\ldots,x_n)$ is weakly regular.
  \item[(c)] The predicate $P(x_1,\ldots,x_n)\wedge Q(x_1,\ldots,x_n)$ is not necessarily weakly regular.
  \item[(d)] The predicate $(\exists x_1)P(x_1,\ldots,x_n)$ is weakly regular.
  \item[(e)] The predicate $(\forall x_1)P(x_1,\ldots,x_n)$ is not necessarily weakly regular.
  \item[(f)] If $n=2$, the predicate $(\exists x_1)P(x_1,x_2)$ is regular.
\end{enumerate}
\end{proposition}

\begin{proof}

 Claim (f) is shown in \cite{Shapiro}.

 We first prove (b).  Given two $n$-variable weakly regular languages, let $M$ and $N$ be corresponding non-deterministic semi-sorted asynchronous automata (SAA's), with state sets $S_1,\ldots,S_n$ and $T_1,\ldots,T_n$ respectively.  Then the disjoint union of their state diagrams, with state sets $S_1\cup T_1, \ldots, S_n\cup T_n$, is another SAA that accepts the union of the two weakly regular languages.

  For (c), consider the two-variable languages $$L_1=\{(x^n y x^m,x^k y x^n)\}$$ and $$L_2=\{(x^n y x^m,x^n y x^k)\}.$$    First, note that each of $L_1$ and $L_2$ is a weakly regular language; in fact, they are quasi-regular, with the state diagram of a semi-sorted asynchronous automaton accepting $L_1$ shown in Figure \ref{sorted3}.  We can easily modify the diagram to see that $L_2$ is quasi-regular as well.

\begin{figure}[h]
\begin{center}
\includegraphics{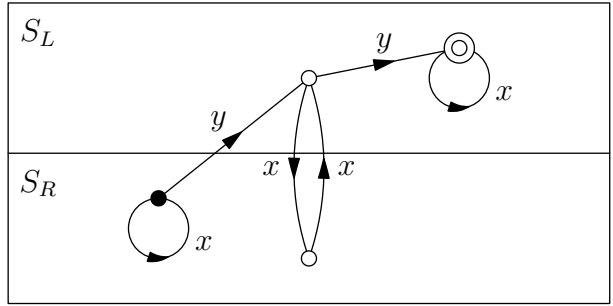}
\caption{\label{sorted3}A semi-sorted asynchronous automaton accepting the language $\{(x^n y x^m,x^k y x^n)\}$.}
\end{center}
\end{figure}

  Now, assume for contradiction that the language $L_1\cap L_2=\{(x^nyx^m, x^nyx^n)\}$ is weakly regular.  By (f), it follows that the one-variable language $\{(x^nyx^n)\}$ is regular.  But the pumping lemma shows that this cannot be regular, and so we have a contradiction.  This proves (c).

  We can now prove (a).  Suppose that the complement of any weakly regular language is weakly regular.  Then using the identity $P\wedge Q=\neg((\neg P)\vee (\neg Q))$ and the fact that weakly regular languages are closed under union, we have that they are closed under intersection, contradicting (c).

  For (d), let $M$ be a SAA accepting the language defined by $P$, with state sets $S_1,\ldots, S_n$.  Then we can replace all arrows starting in the state set $S_1$ corresponding to $x_1$ by $\epsilon$-arrows and merge the states of $S_1$ with $S_2$ to obtain a SAA that accepts the language defined by $(\exists x_1) P(x_1,\ldots,x_n)$.

  Finally, for part (e), we use the languages $L_1$ and $L_2$ defined above.  Since they are quasi-regular, their complements $\overline{L_1}$ and $\overline{L_2}$ are quasi-regular as well.  Thus, the language $L:=\overline{L_1}\cup\overline{L_2}$ is weakly regular by (b).  However, its complement, $\overline{L}=\overline{\overline{L_1}\cup\overline{L_2}}=L_1\cap L_2$, is not weakly regular, as above.

  By Theorem \ref{Bridge}, there is a quasi-regular predicate $R(u,v,w)$ for which $(\exists u)R(u,v,w)$ defines the language $L$.  Thus, the negation of the statement, $\neg (\exists u) R(u,v,w)$ is not weakly regular.  This statement can be rewritten as $(\forall u)\neg R(u,v,w)$.  Since quasi-regular predicates are closed under negation, it follows that there is a quasi-regular (and hence weakly regular) predicate $P(u,v,w)$, namely, $\neg R(u,v,w)$, for which $(\forall u) P(u,v,w)$ is not weakly regular.  This completes the proof.
\end{proof}


\section{Asynchronous automatic structures of finitely presented groups}\label{Groups}

\subsection{Background}

 We first give some background on finitely presented and automatic groups, following the conventions and terminology in \cite{Epstein}.

 Let $A$ be a finite set along with a pairing of its elements, so that paired elements are called \textit{inverses} of each other.  If $x\in A$, we write $x^{-1}\in A$ to denote the formal inverse of $x$ in $A$.  (Note that an element may be its own inverse.)  Then the \textit{free group} on $A$, denoted $F(A)$, is the group under concatenation of all words over $A$ that contain no adjacent inverse generators.  

  A \textit{finite presentation} of a group $G$ consists of a finite inverse-closed set $A\subset G$ called the \textit{generating set} or the set of \textit{generators}, along with a finite set $R\subset A^\ast$ called the set of \textit{relators}, and such that if $N$ denotes the smallest normal subgroup of $F(A)$ containing $R$, then $F(A)/N=G$.  In this case, we write $G=\langle A \mid R\rangle$. 

  Given a finite presentation $G=\langle A\mid R\rangle$ and a word $w\in A^\ast$, we write $\widehat{w}$ to denote the element of $G$ that $w$ represents, that is, when we interpret concatenation as group multiplication.

\begin{definition}
  Let $G=\langle A \mid R\rangle$ be a finitely presented group.  A \textit{automatic structure} for the presentation is a finite state automaton $W$, called the \textit{word acceptor}, along with \textit{multiplier automata} $M_x$ for each $x\in A\sqcup \{\epsilon\}$, such that the following hold:
\begin{itemize}
  \item The language accepted by $L$ represents every element of the group, that is, $\{\widehat{w}\mid w\in L\}=G$.  
  \item For each $x\in A$, $M_x$ is a finite state automaton that accepts precisely those pairs of words $(w_1,w_2)\in L\times L$ for which $\widehat{w_1x}=\widehat{w_2}$.
  \item $M_\epsilon$ is a finite state automaton that accepts precisely those pairs of words $(w_1,w_2)\in L\times L$ for which $\widehat{w_1}=\widehat{w_2}$.
\end{itemize}
\end{definition}

  It is known that if a group has an automatic structure with respect to one set of generators, then it has an automatic structure with respect to every set of generators \cite{Epstein}.  Thus, if a group has a finite presentation with an automatic structure, it is said that the group is \textit{automatic}.

Epstein, et.~al \cite{Epstein} gave a similar definition of an asynchronous automatic group.  

\begin{definition}
  Let $G=\langle A \mid R\rangle$ be a finitely presented group.  An \textit{asynchronous automatic structure} for the presentation is a finite state automaton $W$, called the word acceptor, along with asynchronous multiplier automata $M_x$ for each $x\in A\sqcup \{\epsilon\}$, such that the following hold:
\begin{itemize}
  \item The language accepted by $L$ represents every element of the group, that is, $\{\widehat{w}\mid w\in L\}=G$.  
  \item For each $x\in A$, $M_x$ is an asynchronous automaton that accepts precisely those pairs of words $(w_1,w_2)\in L\times L$ for which $\widehat{w_1x}=\widehat{w_2}$.
  \item $M_\epsilon$ is an asynchronous automaton that accepts precisely those pairs of words $(w_1,w_2)\in L\times L$ for which $\widehat{w_1}=\widehat{w_2}$.
\end{itemize}
\end{definition}

As in the synchronous case, if a group has an asynchronous automatic structure with respect to one set of generators, then it has an asynchronous automatic structure with respect to every set of generators \cite{Epstein}.  Thus, if a group has a finite presentation with an asynchronous automatic structure, we say that the group is \textit{asynchronous automatic}.

\begin{remark}
While every asynchronous automatic group is automatic, the class of asynchronous automatic groups is strictly larger.  In particular, for $p\neq q$, the Baumslag-Solitar group $G_{p,q}=\langle\{x,y\}/\{yx^py^{-1}x^{-q}\}\rangle$ is asynchronous automatic, but not automatic.
\end{remark}

It would seem natural to go on to define a non-deterministic asynchronous automatic group in a similar fashion.  However, in \cite{Shapiro}, Shapiro proved that any such group also admits a (deterministic) asynchronous automatic structure.  For this reason, we work with determinstic asynchronous automatic structures throughout.

One particular type of asynchronous automaton, defined in \cite{Epstein}, will be useful in our study of asynchronous automatic groups.

\begin{definition}
  An asynchronous automaton is \textit{bounded} if there is an integer constant $k$ such that the automaton never reads more than $k$ letters in a row from any of its tapes.  We say that an asynchronous automatic structure is \textit{bounded asynchronous} if each of its multiplier automata are bounded.
\end{definition}

\begin{theorem}[Epstein, et.~al, \cite{Epstein}]\label{Bounded}
  Let $G$ be a group with an asynchronous automatic structure given by an alphabet $A$, a word acceptor $W$, and multiplier automata $M_x$ for $x\in A\sqcup \{\epsilon\}$.  Then $G$ has a boundedly asynchronous automatic structure over $A$, with a language that is a subset of $L(W)$.  Moreover, there is an effective procedure for constructing the boundedly asynchronous automatic structure from the original structure, and this procedure does not depend on $G$.
\end{theorem}

\subsection{Recovering a group from a given set of automata}

Much work has been done on understanding which groups have an automatic structure.  In parallel, the problem has been investigated in reverse: given a set of automata over an alphabet $A$, how can one tell if they are the (asynchronously) automatic structure of some finitely presented group?

In \cite{Epstein}, Epstein, et.~al answered this question in the case of synchronous automatic structures.  In particular, they gave a set of $13$ axioms, each of which are statements about the automata $W$, $M_x$, such that the automata are the automatic structure of some group if and only if all $13$ axioms are satisfied.  Moreover, these axioms are decidable predicates (that is, there is an algorithm that returns $1$ if the predicate is true and $0$ if the predicate is false), and they give an algorithm for finding a finite presentation of the group when it exists.

We now provide a similar result in the case of asynchronous automata.  In light of Theorem \ref{Bounded}, we only consider the case in which the multiplier automata are bounded.

\begin{remark}
While every asynchronous automatic group is automatic, the class of asynchronous automatic groups is strictly larger.  In particular, for $p\neq q$, the Baumslag-Solitar group $G_{p,q}=\langle\{x,y\}/\{yx^py^{-1}x^{-q}\}\rangle$ is asynchronous automatic, but not automatic.
\end{remark}

\begin{theorem}\label{Axioms}
Let $A$ be a finite alphabet, let $W$ be a finite state automaton accepting the regular language $L=L(W)$, and let $\{M_x\}$ be a collection of two-tape boundedly asynchronous automata for each $x\in A\sqcup\{\epsilon\}$.  Then there is a group $G$ for which $W$ and $\{M_x\}$ form an asynchronous automatic structure for $G$ if and only if the following axioms hold:

\begin{enumerate}
  \item $(\exists w)(w\in L).$  

  \item For each $x\in A\cup \{\epsilon\}$, $(\forall w,v)((w,v)\in L_x\implies w\in L \AND v\in L).$

  \item $(\forall w)(w\in L \implies (w,w)\in L_\epsilon).$
 
  \item $(\forall u,v) ((u,v)\in L_\epsilon \implies (v,u)\in L_\epsilon).$

  \item $(\forall u,v,w) (((u,v)\in L_\epsilon \AND (v,w)\in L_\epsilon) \implies (u,w)\in L_\epsilon).$

  \item For each $x\in A$, $(\forall u)(u\in L\implies (\exists v)((u,v)\in L_x)).$

  \item For each $x\in A$, $(\forall u,v,w)(((u,v)\in L_x \AND (v,w)\in L_\epsilon)\implies (u,w)\in L_x).$

  \item For each $x\in A$, $(\forall u,v,w)(((u,v)\in L_\epsilon\AND (u,w)\in L_x) \implies (v,w)\in L_x).$

  \item For each $x\in A$, $(\forall v)(v\in L\implies (\exists u)((u,v)\in L_x).$

  \item For each $x\in A$, $(\forall u,v,w)(((u,v)\in L_x \AND (u,w)\in L_\epsilon)\implies (w,v)\in L_x).$

  \item For each $x\in A$, $(\forall u,v,w)(((u,v)\in L_\epsilon\AND (w,u)\in L_x) \implies (w,v)\in L_x).$

  \item For a word $w=\sigma_1\ldots\sigma_n$ with each $\sigma_i\in A$, we write $[v]\varphi_w=[u]$ to denote the statement $$(\exists v_1,\ldots,v_{n-1})((v,v_1)\in L_{\sigma_1} \AND (v_1,v_2)\in L_{\sigma_2}\AND \cdots \AND (v_{n-1},u)\in L_{\sigma_n}).$$  Then 
$$(\forall u,w,w')((uw\in L \AND uw' \in L)\implies (\forall v)([v]\varphi_w=[uw] \iff [v]\varphi_{w'}=[uw'])).$$
   
  \item Let $c$ be the maximum number of states in any of $W$, $M_\epsilon$, and $M_x$, and let $k$ be the largest boundedness factor of any $M_x$ or $M_\epsilon$.  For each word $w$ over $A$ of length at most $2c+2k$, $$(\exists u)([u]\varphi_w=[u])\implies(\forall u)([u]\varphi_w=[u]).$$

\end{enumerate}
\end{theorem}

\begin{remark}
  If we are given a collection of (possibly) \textup{unbounded} asynchronous automata, we can first apply the algorithm given by Theorem \ref{Bounded}, check if the resulting automata are bounded (by looking for loops entirely contained in the left or right state set) and then apply Theorem \ref{Axioms}.  Thus, if Axioms $1$-$13$ are decidable for bounded asynchronous automata, then the problem of recovering a group from (possibly unbounded) asynchronous automata is decidable as well.
\end{remark}

It is easily verified that a bounded asynchronous automatic structure of a group must satisfy each of the axioms of Theorem \ref{Axioms}.  In order to prove the reverse direction, we first prove several lemmas.  

\begin{notation}
Throughout the remainder of this section, let $A$ be a finite alphabet, $L=L(W)$ and $L_x=L(M_x)$ for each $x\in A\sqcup\{\epsilon\}$, where $W$ is a finite state automaton and each $M_x$ is a bounded asynchronous automaton over $A$ such that $L$ and $L_x$ satisfy Axioms $1$-$13$.

By Axioms $3$-$5$, we may partition $L$ into a set of equivalence classes $X$ under the equivalence relation $u\sim v$ if and only if $(u,v)\in L_{\epsilon}$.  We write $[u]$ to denote the equivalence class of a word $u\in L$.
\end{notation}

\begin{lemma}
  For each $x\in A$, there exist unique invertible maps $\varphi_x:X\rightarrow X$ (acting on the right) such that for any $u,v\in L$, $[u]\varphi_x=[v]$ if and only if $(u,v)\in L_x$.
\end{lemma}

\begin{proof}
  Fix $x\in A$.  For each $u$, we can use Axiom $6$ to choose a word $v\in L$ such that $(u,v)\in L_x$, and define a map $s_x:L\rightarrow L$ by $us_x=v$.  Then by Axiom $7$, the induced map $s'_x:L\rightarrow X$ by $us'_x=[v]$ is independent of our original choices of $v$.  By Axiom $8$, if $u$ and $w$ are in the same equivalence class then $s'_x$ maps them to the same equivalence class $[v]$, and so $s'_x$ restricts to a map $\varphi_x:X\rightarrow X$ having the desired property.  Finally, Axiom $7$ shows that this map is unique.

  Axioms $9$-$11$ similarly define maps $\mu_x:X\rightarrow X$ for which $[v]\mu_x=[u]$ if and only if $(u,v)\in L$.  Then $\mu_x=\varphi_x^{-1}$ for each $x$, and so we see that the maps $\varphi_x$ are invertible, as desired. 
\end{proof}

This lemma, combined with Axiom $12$, allows us to extend the notion of an equivalence class to the prefix closure of $L$, which we denote by $\overline{L}$, as follows.  For each prefix $u$ of a word $uw$ in $L$, define $[u]=[uw]\varphi_{w}^{-1}$.  Axiom $12$ shows that this is a well-defined equivalence class.

\begin{notation}
  If $w=x_1\cdots x_n$ is a word over $A$, we define $\varphi_w=\varphi_{x_1}\varphi_{x_2}\cdots \varphi_{x_n}$ where $\varphi_{x^{-1}}$ is defined to be $\varphi_x^{-1}$.  
\end{notation}

\begin{lemma}\label{PrefixAction} 
We have $[\epsilon]\varphi_u=[u]$ for any prefix $u\in \overline{L}$.
\end{lemma}

\begin{proof}
  Notice that if $uw\in L$, then $[\epsilon]\varphi_{uw}=[uw]$ by the definition of the extension of $\sim$ to prefixes, and so $[\epsilon]\varphi_{u}=[uw]\varphi_{w}^{-1}=[u]$.
\end{proof}

Finally, define $H$ to be the group generated by the maps $\varphi_x$ under composition.  Then $H$ acts on $X$ on the right.  We wish to show that this action is transitive and free, for we can then identify the elements of $H$ with the elements of $X$.

\begin{lemma}\label{Transitive}
  The action of $H$ on $X$ is transitive.
\end{lemma}

\begin{proof}
  Given two equivalence classes $[u]$ and $[v]$ where $u,v\in L$, we note that 
$$[u]\varphi_{u}^{-1}\varphi_v=[\epsilon]\varphi_v=[v],$$ 
and so each equivalence class is mapped to every other under the group action.
\end{proof}

To show that the action is free, we first prove the following lemma.

\begin{lemma}\label{Loops}
  Suppose $x\in A$ and $u,u'$ are words in $L$ such that $[\epsilon]\varphi_u\varphi_x\varphi_{u'}^{-1}=[\epsilon]$.  Then $\varphi_u\varphi_x\varphi_{u'}^{-1}$ is the identity in $H$.
\end{lemma}

\begin{proof}
   First note that the assumption implies $[u]\varphi_x=[u']$, since $u,u'\in L$, and so by the definition of the maps $\varphi_x$ we have that $(u,u')$ is accepted by the asynchronous automaton $M_x$.  Recall that this assigns a unique shuffle to $(u,u')$ as well.

For each $t\ge 0$, define $u\langle t\rangle$ to be the word formed by the first $t$ blocks of consecutive letters of $u$ in the shuffle of $u$ and $u'$.  By our assumption, each block is of length at most $k$.  Let the blocks of $u$ be $B_1,\ldots,B_n$ and those of $u'$ be $B_1',\ldots,B_n'$ (where one of $B_n$ or $B_n'$ may be the empty string, depending on the shuffle, and all other blocks are nonempty and of length at most $k$).  In this notation, we have $u\langle t \rangle=B_1\cdots B_t$ for $t\le n$ and $u\langle t \rangle=B_1\cdots B_n$ for $t>n$.

  Note that for each $t$, at some point in the path accepting $(u,u')$ in $M_x$ we have traversed a shuffle corresponding to $(u\langle t \rangle, u'\langle t \rangle)$.  By removing possible loops from this path, the shortest path from this point to the accept state is less than the number of states of $M_x$, so it is at most $c-1$ where $c$ is the maximum number of states of any of the automata $M_x$ or $M_\epsilon$.  Thus there exist words $w_t$ and $w'_t$ of total length at most $c-1$ for which $(u\langle t \rangle w_t,u'\langle t \rangle w'_t)\in L_x$.

  Consider the word $$r_t:=w_tx{w'}_{t}^{-1}B'_{t+1}w'_{t+1}x^{-1}w_{t+1}^{-1}B_{t+1}^{-1}.$$  This has length at most $c-1+1+k+c-1+1+k=2c+2k$.  We wish to show it fixes some element of $X$, in order to apply Axiom $13$.   We have
\begin{eqnarray*}
  [u\langle t\rangle] \varphi_{r_t} &=& [u \langle t \rangle] \varphi_{w_t x w_{t}'^{-1} B_{t+1}' w_{t+1}' x^{-1} w_{t+1}^{-1} B_{t+1}^{-1}} \\
                                   &=& [u\langle t\rangle w_t] \varphi_{x}\varphi_{w_{t}'^{-1}B_{t+1}'w_{t+1}'x^{-1}w_{t+1}^{-1}B_{t+1}^{-1}} \\
				   &=& [u'\langle t\rangle w_{t}'] \varphi_{w_{t}'^{-1}B_{t+1}' w_{t+1}' x^{-1}w_{t+1}^{-1}B_{t+1}^{-1}} \\
				   &=& [u'\langle t\rangle] \varphi_{B_{t+1}'w_{t+1}'x^{-1}w_{t+1}^{-1}B_{t+1}^{-1}} \\
				   &=& [u'\langle t+1\rangle] \varphi_{w_{t+1}'x^{-1}w_{t+1}^{-1}B_{t+1}^{-1}} \\
				   &=& [u'\langle t+1\rangle w_{t+1}'] \varphi_{x^{-1}}\varphi_{w_{t+1}^{-1}B_{t+1}^{-1}} \\
				   &=& [u\langle t+1\rangle w_{t+1}] \varphi_{w_{t+1}^{-1}B_{t+1}^{-1}} \\
				   &=& [u\langle t+1\rangle] \varphi_{B_{t+1}^{-1}} \\
				   &=& [u\langle t\rangle]
\end{eqnarray*}
and so, by Axiom $13$, $\varphi_{r_t}$ is the identity in $H$.  It follows that for each $t= 1,2,\ldots,n-1$, we have 
$$\varphi_{w_txw_{t}'^{-1}B_{t+1}' w_{t+1}' x^{-1}w_{t+1}^{-1}}=\varphi_{B_{t+1}}$$ 
and similarly $$\varphi_{B'_1w_1'^{-1}x^{-1}w_1^{-1})}=\varphi_{B_1}.$$

Multiplying these $n$ equations together, we find $$\varphi_{B_1'B_2'\cdots B_n'x^{-1}}=\varphi_{B_1B_2\cdots B_n},$$
so $\varphi_{u'}\varphi_{x}^{-1}=\varphi_{u}$, and thus $\varphi_{u}\varphi_x\varphi_{u'}^{-1}$ is the identity, as desired.
\end{proof}

\begin{lemma}\label{Free}
  The action of $H$ on $X$ is free.
\end{lemma}

\begin{proof}
   Note that it suffices to show that, for all words $v$, 
$$[\epsilon]\varphi_v=[\epsilon] \implies (\forall u)([u]\varphi_v=[u]).$$

For, if this holds, then if there is $u$ in the prefix closure of $L$ such that $[u]\varphi_w=[u]$, then $[\epsilon]\varphi_{u}\varphi_w\varphi_{u}^{-1}=[\epsilon]$, so $\varphi_{u}\varphi_w\varphi_{u}^{-1}=\id$, and hence $\varphi_w=\varphi_{u}^{-1}\varphi_{u}=\id$ as well.

Let $w$ be an arbitrary word that fixes the basepoint $[\epsilon]$, that is, $[\epsilon]\varphi_w=[\epsilon]$.  (Note that such a word $w$ must exist by Lemma \ref{Transitive}.)  Write $w$ in a reduced form $x_1x_2\cdots x_n$ where each $x_i\in A\cup A^{-1}$, and there are no pairs of consecutive letters of the form $xx^{-1}$ or $x^{-1}x$.

By our definition of the equivalence classes on the prefix closure of $L$, each equivalence class can be represented by an element of $L$ itself, so for each $t=1,\ldots,n-1$, there is some word $u_t\in L$ such that $$[u_t]=[\epsilon]\varphi_{x_1\cdots x_t}=[\epsilon]\varphi_{w_t}=[w(t)].$$  Also set $u_0=u_n=\epsilon$.

Then 
\begin{eqnarray*}
  [\epsilon]\varphi_{u_t}\varphi_{x_{t+1}}\varphi_{u_{t+1}}^{-1} &=& [u_t]\varphi_{x_{t+1}}\varphi_{u_{t+1}}^{-1} \\
       &=& [w(t)]\varphi_{x_{t+1}}\varphi_{u_{t+1}}^{-1} \\
       &=& [w(t+1)]\varphi_{u_{t+1}}^{-1} \\
       &=& [u_{t+1}]\varphi_{u_{t+1}}^{-1} \\
       &=& [\epsilon].
\end{eqnarray*}

By Lemma \ref{Loops}, we have that $\varphi_{u_{t}}\varphi_{x_{t+1}}\varphi_{u_{t+1}}^{-1}$ is the identity in $H$.  Multiplying these together over $t=0,\ldots,n-1$ yields $$\varphi_{u_0}\varphi_{x_1}\varphi_{u_1}^{-1}\varphi_{u_1}\varphi_{x_2}\cdots \varphi_{x_n}\varphi_{u_n}^{-1}=\id,$$ which simplifies to $$\varphi_{x_1\cdots x_n}=\id.$$   Thus $\varphi_{w}=\id$, as desired.
\end{proof}

We now prove Theorem \ref{Axioms}.

\begin{proof}
  Let $G=H$ be the group having the transitive and free action on $X$ described in Lemma \ref{Free}.  Since the action is transitive and free, we may identify the elements of $G$ bijectively with the elements of $X$, as follows.  Identify the identity element of $G$ with $[\epsilon]$, and for each $g\in G$, identify $g$ with $[\epsilon]g$.  Then transitivity gives that this identification is surjective, and freedom gives that this identification is injective, and thus it is a well defined bijection.

  It follows that the action of $G$ on $X$ is isomorphic to the action of $G$ on itself by right multiplication.  Therefore, $M_x$ accepts precisely the pairs of words representing elements of $G$ that differ by the generator $x$ in the Cayley graph for each $x$, and $M_\epsilon$ accepts the pairs of words in $L$ representing the same element of $G$.  It follows that $L$, $M_x$ form an asynchronous automatic structure for $G$, as desired.
\end{proof}

\subsection{Decidability}

In this section, we assume basic familiarity with the concept of decidable predicates.  For a thorough introduction to this topic, see \cite{Cutland}.

In \cite{Epstein}, Epstein, et.~al gave a set of axioms, each of which are first-order sentences involving regular predicates, that allow one to recover a group from a set of synchronous automata if all the axioms are decidably true, or determine that there is no such group if one of the axioms is decidably false.  In Theorem \ref{Axioms}, we have given a similar set of axioms for recovering a group from a set of asynchronous automata, each of which are first-order sentences involving quasi-regular predicates.  Since regular predicates are closed under first order operations, the axioms for synchronous automata are regular and therefore decidable, but it is less clear whether the axioms of Theorem \ref{Axioms} are quasi-regular, or even decidable.  In this section, we investigate the decidability of Axioms $1$-$13$.

Adopting the terminology in \cite{Cutland}, we say that a statement is \textit{partially decidable} if there is an algorithm that halts and outputs true if the statement is true, and does not halt if the statement is false.

\begin{theorem}\label{PartiallyDecidable}
  Let $P(W,\{M_x\})$ denote the statement: ``The finite state automaton $W$ over the alphabet $A$ and asynchronous automata $M_x$ over $A$, one for each $x\in A\sqcup\{\epsilon\}$, do not form the asynchronous automatic structure of any finitely presented group."  Then $P$ is partially decidable.
\end{theorem}

We first show that Axioms $1$, $2$, $6$, and $9$ are decidable.

\begin{lemma}\label{1269}
  Let $W$ be a finite state automaton over $A$, and let $M_x$ for each $x\in A\cup\{\epsilon\}$ be asynchronous automata over $A$.  Then Axioms $1$, $2$, $6$, and $9$ are decidable predicates.
\end{lemma}

\begin{proof}
  Note that Axiom $1$ is a regular predicate and is therefore decidable.

  For Axiom $2$, we can simplify the statement as follows:
\vspace{12pt}

$(\forall w_1,w_2) ((w_1,w_2)\in L_x) \implies (w_1\in L\wedge w_2\in L))$

$\neg(\exists w_1,w_2) [((w_1,w_2)\in L_x) \wedge (w_1\not\in L\vee w_2\not\in L)]$

$\neg(\exists w_1,w_2) [(((w_1,w_2)\in L_x) \wedge (w_1\not\in L)) \vee (((w_1,w_2)\in L_x) \wedge (w_2\not\in L))]$

$\neg[(\exists w_1,w_2)(((w_1,w_2)\in L_x) \wedge (w_1\not\in L)) \vee (\exists w_1,w_2)((w_1,w_2)\in L_x \wedge (w_2\not\in L))]$

$\neg\left\{\left[ (\exists w_1)((\exists w_2)((w_1,w_2)\in L_x)\wedge (w_1\not\in L)) \right] \vee \left[ (\exists w_2)((\exists w_1)((w_1,w_2)\in L_x)\wedge (w_2\not\in L)) \right]\right\}$
\vspace{10pt}

In the last formulation of the statement above, the smaller statements $(\exists w_2)((w_1,w_2)\in L_x)$ and $(\exists w_1)((w_1,w_2)\in L_x)$ define regular languages in the variables $w_1$ and $w_2$ respectively, by part (f) of Proposition \ref{QuasiClosure}.  The statements $w_1\not\in L$ and $w_2\not\in L$ are regular as well, since regular predicates are closed under negation.  Thus we have rewritten the original statement as a first-order statement involving regular predicates, which is regular and hence decidable.

Axioms $6$ and $9$ can similarly be stated in terms of regular predicates.  This completes the proof.
\end{proof}

We now prove Theorem \ref{PartiallyDecidable}.  

\begin{proof}
  By Lemma \ref{1269}, the negations of Axioms $1$, $2$, $6$, and $9$ are decidable, and hence partially decidable.

  We now show that the negation of Axiom $3$, that is,
$$(\exists w)(w\in L \AND (w,w)\not\in L_\epsilon),$$
is partially decidable.  Indeed, we can order the words in $A^\ast$ in length-lexicographic order (with respect to some ordering of $A$), and check in order if each satisfies $w\in L \AND (w,w)\not\in L_\epsilon$ by passing $w$ and $(w,w)$ through the automata $W$ and $M_\epsilon$, respectively.  When we reach a word $w$ that satisfies the two conditions, we stop, and otherwise we check the next word in the length-lexicographic ordering.  This procedure halts if the statement is true, and runs indefinitely if it is false, as desired.

  This argument can be easily modified to show that Axioms $4$, $5$, $7$, $8$, $10$, and $11$ are partially decidable.

  We next show that the statement ``Either Axiom $9$ is false or Axiom $12$ is false'' is partially decidable.  Consider the natural product ordering of $(A^{\ast})^4$ that arises from the length-lexicographic ordering of $A^\ast$.  We apply the following procedure to the $4$-tuples $(u,w,w',v)$ in $(A^{\ast})^4$ in order.  We use $W$ to check if $uw$ and $uw'$ are in $L$.  If not, we go on to the next $4$-tuple.

If $uw$ and $uw'$ are both in $L$, then we check the validity of each of the statements $[v]\varphi_w=[uw]$ and $[v]\varphi_{w'}=[uw']$.  Let $\sigma_1\sigma_2\ldots\sigma_n$ be the letters of $w$.  We first check if there is some $v_1$ for which $(v,v_1)\in L_{\sigma_1}$.  (We can check this since the projection of a two-variable quasi-regular language is regular.)  If there is no such $v_1$, we stop; this proves the negation of Axiom $9$.  If there is such a $v_1$, we can choose one by searching through all finite paths starting at the start state in order of length until we come across the first shuffle that spells $v$ in the left state set.  We now apply the same procedure to either choose some $v_2$ such that $(v_1,v_2)\in L_{\sigma_2}$, or halt if no such $v_2$ exists.  

Continuing in this fashion, if we have chosen a word $v_i$, we determine whether there is some $v_{i+1}$ with $(v_i,v_{i+1})\in L_{\sigma_{i+1}}$ for $i\le n-1$.  If at the $n$th step we obtain a word $v_n$, we check if $(v_n,uw)\in L_{\epsilon}$.  If so, the statement $[v]\varphi_w=[uw]$ is true, and if not, it is false.  We similarly check the validity of $[v]\varphi_{w'}=[uw']$.  If both are true or both are false, we go on to the next $4$-tuple.  Otherwise, we halt, as this proves the negation of Axiom $12$.  Thus, the statement ``Either Axiom $9$ is false or Axiom $12$ is false'' is partially decidable.

  We next show that ``Either Axiom $9$ is false or Axiom $13$ is false'' is partially decidable.  For each fixed $w$ with length at most $2c+2k$, we check all $u$ in the prefix closure of $L$ to determine if $[u]\varphi_w=[u]$, with a procedure that halts if Axiom $9$ is false, as before.  If we find two strings $u,v$ such that the $[u]\varphi_w=[u]$ but $[v]\varphi_w\neq[v]$, we halt; Axiom $13$ is false.  Otherwise, our procedure runs indefinitely.

  Finally, by Theorem \ref{Axioms}, the statement $P(W,\{M_x\})$ is equivalent to the statement: 
\begin{quote}
``Axiom 1 is false or Axiom 2 is false or ... or Axiom 13 is false,''
\end{quote}
which can be rewritten as 
\begin{quote}
``Axiom 1 is false or Axiom 2 is false or ... or (Axiom 9 is false or Axiom 12 is false) or (Axiom 9 is false or Axiom 13 is false).''
\end{quote}
 The latter is partially decidable, as we can run each of our above procedures in parallel.
\end{proof}

\section{Future Work}

  It remains to be shown whether all of the Axioms of Theorem \ref{Axioms} are decidable.  To do so, it would be useful to further understand the closure properties of quasi-regular and weakly regular predicates, as either can be used to define asynchronous automatic groups. (See \cite{Shapiro})  

  We have shown that quasi-regular languages are closed under complementation but not under union, and weakly regular languages are closed under union but not under complementation.  Thus, it may also be of interest to investigate intermediate classes of languages in order to find one that is closed under both complementation and union.  For instance, the class of all finite unions of quasi-regular languages is larger than the class of quasi-regular languages and smaller than that of weakly regular languages, and it is closed under union (however, it is not closed under complementation). 

  Finally, we note that Rubin \cite{Rubin} defined a generalized notion of quantifiers, and classified the unary quantifiers that preserve regularity.  It would be of interest to study which generalized quantifiers preserve quasi-regular and weakly regular predicates.

\section{Acknowledgments}

This research was done at MIT through the Undergraduate Research Opportunities Program.

I thank my supervisor, Mia Minnes, for her teaching and guidance throughout the course of this research.  I also thank the Colorado State University mathematics colloquium for the opportunity to give a talk on this research.  Finally, I thank Ken G.~Monks, Ken M.~Monks, Paul Christiano, and Rishi Gupta for numerous helpful conversations and for their support along the way.

\end{document}